\documentclass{aims_newer}
\usepackage{amsfonts}
\usepackage{amssymb}
\usepackage{mathrsfs}
  \usepackage{paralist}
  \usepackage{graphics} 
  \usepackage{epsfig} 
\usepackage{graphicx}  \usepackage{epstopdf}
 \usepackage[colorlinks=true]{hyperref}
\hypersetup{urlcolor=blue, citecolor=red}

     \def\DOI{10.3934/xx.xx.xx.xx}

 \def \ProbUnion{P\left(\bigcup_{i=1}^N A_i\right)}
 
 \def \lKAT{\ell_{\textrm{KAT}}}
 
 \def \lGK{\ell_{\textrm{GK}}}
 \def \bSigma{\boldsymbol{\Sigma}}
 \def \balpha{\boldsymbol{\alpha}}
 \def \bone{\boldsymbol{1}}
 \def \st{\textrm{s.t. }}

 \def \bc{\boldsymbol{c}}

 \def \ProbUnion{P\left(\bigcup_{i=1}^N A_i\right)}

 \def \lNEWI{\ell_{\textrm{YAT-I}}}
 \def \lNEWII{\ell_{\textrm{YAT-II}}}
 \def \lNEWIII{\ell_{\textrm{NEW-I}}}
 \def \lNEWIV{\ell_{\textrm{NEW-II}}}

 \def \uNEWIV{\hbar_{\textrm{NEW-I}}}
 \def \uNEWV{\hbar_{\textrm{NEW-II}}}

 \def \lKAT{\ell_{\textrm{KAT}}}
 \def \lPG{\ell_{\textrm{PG}}}

 \def \lDC{\ell_{\textrm{DC}}}
 \def \lGK{\ell_{\textrm{GK}}}
 
\newtheorem{theorem}{Theorem}[section]
\newtheorem{corollary}{Corollary}

\newtheorem{lemma}[theorem]{Lemma}

\theoremstyle{definition}

\newtheorem{remark}{Remark}

\title[On Bounding the Union Probability] 
      {On Bounding the Union Probability Using Partial Weighted Information}

\author[Jun Yang and Fady Alajaji and Glen Takahara]{}


 \email{jun@utstat.toronto.edu}
 \email{fady@mast.queensu.ca}
 \email{takahara@mast.queensu.ca}

\thanks{This work was supported in part by NSERC of Canada.
Parts of this work were presented at the 2015 IEEE International Symposium on Information
Theory, Hong Kong, June~2015.}

\begin{document}
\maketitle

\centerline{\scshape Jun Yang}
\medskip
{\footnotesize
 \centerline{Department of Statistical Sciences}
   \centerline{University of Toronto}
   \centerline{Toronto, ON M5S3G3, Canada}
} 

\medskip

\centerline{\scshape Fady Alajaji and Glen Takahara}
\medskip
{\footnotesize
 \centerline{Department of Mathematics and Statistics}
   \centerline{Queen's University}
   \centerline{Kingston ON K7L3N6, Canada}
}

\bigskip


\begin{abstract}

Effective bounds on the union probability are well known to be beneficial in the analysis 
of stochastic problems in many areas, including probability theory, information theory, 
statistical communications, computing and operations research. 
In this work we present new results on bounding the probability of a finite union of events, $\boldmath{\ProbUnion}$, for a fixed positive integer $\boldmath{N}$, using partial information on the events in terms of $\boldmath{\{P(A_i)\}}$ and $\boldmath{\{\sum_j c_jP(A_i\cap A_j)\}}$ where $\boldmath{c_1}$, $\boldmath{\dots}$, $\boldmath{c_N}$ are given weights. We derive two new classes of lower bounds of at most pseudo-polynomial computational complexity. These classes of lower bounds generalize the existing bound in \cite{Kuai2000} and recent bounds in \cite{Yang2014,Yang2014ISIT} and are numerically shown to be tighter in some cases than the Gallot-Kounias bound \cite{Gallot1966,Kounias1968} and the Pr{\'e}kopa-Gao bound \cite{Prekopa2005} which require more information on the events probabilities. 
\end{abstract}

\section{Introduction}
Lower and upper bounds on the union probability $\ProbUnion$ in terms of the individual event probabilities $P(A_i)$'s and the pairwise event probabilities $P(A_i\cap A_j)$'s were actively investigated in the recent past. The optimal bounds can be obtained numerically by solving linear programming (LP) problems with $2^N$ variables (for instance, see \cite{VenezianiUnpublished,Prekopa2005}). Since the number of variables is exponential in the number of events, $N$, some suboptimal but numerically efficient bounds were proposed, such as the algorithmic Bonferroni-type lower/upper bounds in \cite{Kuai2000a,Behnamfar2005}.

Among the established analytical bounds is the Kuai-Alajaji-Takahara lower bound (for convenience, hereafter referred to as the KAT bound) \cite{Kuai2000} that was shown to be better than the Dawson-Sankoff (DS) bound \cite{Dawson1967} and the D.~de Caen (DC) bound \cite{DeCaen1997}. Noting that the KAT bound is expressed in terms of $\{P(A_i)\}$ and only the \emph{sums} of the pairwise event probabilities, i.e., $\{\sum_{j:j\neq i} P(A_i\cap A_j)\}$, in order to fully exploit all pairwise event probabilities, it is observed in \cite{Behnamfar2007,Hoppe2006,Hoppe2009} that the analytical bounds can be further improved algorithmically by optimizing over subsets. Furthermore, in \cite{Prekopa2005}, the KAT bound is extended by using additional partial information such as the sums of joint probabilities of three events, i.e., $\{\sum_{j,l}P(A_i\cap A_j\cap A_l), i=1,\dots,N\}$. Recently, using the same partial information as the KAT bound, i.e., $\{P(A_i)\}$  and  $\{\sum_{j:j\neq i} P(A_i\cap A_j)\}$, the optimal lower/upper bound as well as a new analytical bound which is sharper than the KAT bound were developed by Yang-Alajaji-Takahara in \cite{Yang2014,Yang2014ISIT} (for convenience, these two bounds are respectively referred to as the YAT-I and YAT-II bounds).

In this paper, we extend the existing analytical lower bounds, the KAT bound and the YAT-II bound, and establish two new classes of lower bounds on $\ProbUnion$ using $\{P(A_i)\}$ and $\{\sum_j c_j P(A_i\cap A_j)\}$ for a given weight or parameter vector $\bc=\left(c_1,\dots,c_N\right)^T$. These lower bounds are shown to have at most pseudo-polynomial computational complexity and to be sharper in certain cases than the existing Gallot-Kounias (GK) bound \cite{Gallot1966,Kounias1968} and Pr{\'e}kopa-Gao (PG) bound \cite{Prekopa2005}, although the later bounds employ more information on the events joint probabilities. 

More specifically, we first propose a novel expression for the union probability using given weight vector $\bc$. Then we show using the Cauchy-Schwarz inequality that several existing bounds, such as the bound in \cite{Cohen2004}, the DC bound and the GK bound, can be directly derived from this new expression. Next, we derive two new classes of lower bounds as functions of the weight vector $\bc$ by solving linear programming problems. The existing KAT and YAT-II analytical bounds are shown to be special cases of these two new classes of lower bounds. Furthermore, it is noted that the proposed lower bounds can be sharper than the GK bound under some conditions.

We emphasize that our bounds can be applied to any general estimation
problem involving the probability of a finite union of events. In
particular, they can be applied to effectively estimate and analyze the
error performance of a wide variety of coded or uncoded communication
systems under different decoding techniques (see
\cite{Seguin1998,Kuai2000a,Yousefi2004,Cohen2004,Behnamfar2005,Nguyen2005,Sasson2006,Behnamfar2007,Bettancourt2008,Mao2013,Yang2014ISIT,Ozcelikkale2014}	
and the references therein). Such bounds can also be pertinently useful in the
analysis of asymptotic problems such as the Borel-Cantelli Lemma
and its generalization (e.g., \cite{Erdos1959, Feng2009, Frolov2012, Feng2013}).
Finally, we note that the proposed bounds provide useful tools for
chance-constrained stochastic programs (e.g., see \cite{Prekopa1995,Shapiro2014}) in operations research. More specifically, using partial information of uncertainty, the proposed bounds on the union probability can be applied to formulate tractable conservative approximations of chance-constrained stochastic problems, which can be solved efficiently and produce feasible solutions for the original problems (see, for instance, \cite{Pinter1989,Nemirovski2006,Ben-Tal2009}). An example of such application is the work by \cite{Ahmed2013} on the probabilistic set covering problem with correlations, where the existing KAT bound is used to tackle the case where only partial information on the correlation is available.

The outline of this paper is as follows. In Section \ref{sec_2}, we propose a new expression of the union probability using weight vector $\bc$ and show that many existing bounds can be directly derived from this expression. In Section \ref{sec_3}, we develop two new classes of lower bounds as functions of the weight vector $\bc$ and discuss their connection with the existing bounds, including the KAT bound, the YAT-II bound and the GK bound. Finally, in Section \ref{sec_4}, we compare via numerical examples existing lower bounds with the proposed lower bounds under different choices of weight vectors.

\section{Lower Bounds via the Cauchy-Schwarz Inequality}\label{sec_2}
For simplicity, and without loss of generality, we assume the events $\{A_1,\dots,A_N\}$ are in a finite probability space $(\Omega ,\mathscr{F},P)$, where $N$ is a fixed positive integer. Let $\mathscr{B}$ denote the collection of all non-empty subsets of $\{1,2,\dots,N\}$. Given $B\in\mathscr{B}$, we let $\omega_B$ denote
the atom in $\cup_{i=1}^N A_i$  such that for all $i=1,\dots,N$,
$\omega_B\in A_i$ if $i\in B$ and $\omega_B\notin A_i$ if $i\notin B$
(note that some of these ``atoms'' may be the empty set). For ease of notation, for a singleton $\omega\in\Omega$, we denote $P(\{\omega\})$ by $p(\omega )$ and $P(\omega_B)$ by $p_B$. Since $\{\omega_B: i\in B\}$ is the collection of all the atoms in $A_i$, we have $P(A_i)=\sum_{\omega\in A_i} p(\omega)=\sum_{B\in\mathscr{B}: i\in B} p_B$, and

\begin{equation}\label{ProbUnion_as_sum_of_pB}
	P\left(\bigcup_{i=1}^N A_i\right)=\sum_{B\in\mathscr{B}}p_B.
\end{equation}
Suppose there are $N$ functions $f_i(B), i=1,\dots,N$ such that $\sum_{i=1}^N f_i(B)=1$ for any $B\in\mathscr{B}$ (i.e., for any atom $\omega_B$). If we further assume that $f_i(B)=0$ if $i\notin B$ (i.e., $\omega_B\notin A_i$), we can write
\begin{equation}\label{general_prob_union}
	\begin{split}
		P\left(\bigcup_{i=1}^N A_i\right)&=\sum_{B\in\mathscr{B}}\left(\sum_{i=1}^N f_i(B)\right)p_B
		=\sum_{i=1}^N \sum_{B\in\mathscr{B}: i\in B} \hspace{-0.1in} f_i(B)p_B.
	\end{split}
\end{equation}

Note that if we define the degree of $\omega$, $\deg(\omega)$, to be the number of $A_i$'s that contain $\omega$, then by the definition of $\omega_B$, we have $\deg(\omega_B)=|B|$. Therefore,
\begin{equation}\label{f_def_deg}
	f_i(B)=\left\{ \begin{array}{ll}
		\frac{1}{|B|}=\frac{1}{\deg(\omega_B)} & \textrm{if $i\in B$}\\
		0 & \textrm{if $i\notin B$}
	\end{array} \right.
\end{equation}
satisfies $\sum_{i=1}^N f_i(B)=1$ and (\ref{general_prob_union}) becomes
\begin{equation}\label{traditional_expression_prob_union}
	P\left(\bigcup_{i=1}^N A_i\right)=\sum_{i=1}^N \sum_{B\in\mathscr{B}: i\in B} \frac{p_B}{\deg(\omega_B)}=\sum_{i=1}^N\sum_{\omega\in A_i}\frac{p(\omega)}{\deg(\omega)}.
\end{equation}
Note that many of the existing bounds, such as the DC bound, the KAT bound and the recent bounds in \cite{Yang2014} and \cite{Yang2014ISIT}, are based on (\ref{traditional_expression_prob_union}).

In the following lemma, we propose a generalized expression of (\ref{traditional_expression_prob_union}). To the best of our knowledge this lemma is novel.
\begin{lemma}
	Suppose $\{\omega_B, B\in\mathscr{B}\}$ are all the $2^N-1$ atoms in $\bigcup_i A_i$. If $\bc=(c_1,\dots,c_N)^T\in\mathbb{R}^N$ satisfies
	\begin{equation}\label{condition1}
		\sum_{k\in B} c_k\neq 0 ,\quad \textrm{for all}\quad B\in\mathscr{B}
	\end{equation}
	then we have
	\begin{eqnarray}\label{prob_union}
		P\left(\bigcup_{i=1}^N A_i\right)&=&\sum_{i=1}^N\sum_{B\in\mathscr{B}: i\in B}\frac{c_ip_B}{\sum_{k\in B} c_k} \nonumber \\
		&=&\sum_{i=1}^N\sum_{\omega\in A_i}\frac{c_i p(\omega)}{\sum_{\{k:\omega\in A_k\}}c_k}.
	\end{eqnarray}
\end{lemma}
\begin{proof}
If we define
\begin{equation}\label{f_def_c}
	f_i(B)=\left\{ \begin{array}{ll}
		\frac{c_i}{\sum_{k\in B} c_k} & \textrm{if $i\in B$}\\
		0 & \textrm{if $i\notin B$}
	\end{array} \right.
\end{equation}
where the parameter vector $\bc=(c_1,c_2,\dots,c_N)^T$ satisfies $\sum_{k\in B} c_k\neq 0$ for all $B\in\mathscr{B}$ (therefore $c_i\neq 0, i=1,\dots,N$), then $\sum_i f_i(\omega)=1$ holds and we can get (\ref{prob_union}) from~(\ref{general_prob_union}).
\end{proof}

Note that (\ref{prob_union}) holds for any $\bc$ that satisfies (\ref{condition1}) and is clearly a generalized expression of (\ref{traditional_expression_prob_union}).

\subsection{Relation to the Cohen-Merhav bound by \cite{Cohen2004}}
Let $f_i(B)>0$ and $m_i(\omega_B)$ be non-negative real functions. Then by the Cauchy-Schwarz inequality,
\begin{equation}\label{cauchy_using_m_i}
	\left[\sum_{B: i\in B}f_i(B)p_B\right]\left[\sum_{B: i\in B}\frac{p_B}{f_i(B)}m_i^2(\omega_B)\right]\ge
	\left[\sum_{B: i\in B}p_Bm_i(\omega_B)\right]^2.
\end{equation}
Thus, using (\ref{general_prob_union}), we have
\begin{equation}\label{}
	P\left(\bigcup_{i=1}^N A_i\right)=\sum_{i=1}^N \sum_{B: i\in B}f_i(B)p_B\ge\sum_{i=1}^N\frac{\left[\sum_{B: i\in B}p_Bm_i(\omega_B)\right]^2}{\sum_{B: i\in B}\frac{p_B}{f_i(B)}m_i^2(\omega_B)}.
\end{equation}
If we define $f_i(B)$ by (\ref{f_def_deg}), then the above inequality reduces to
\begin{equation}\label{temp_using_for_below}
	P\left(\bigcup_{i=1}^N A_i\right)\ge\sum_{i=1}^N\frac{\left[\sum_{B: i\in B}p_Bm_i(\omega_B)\right]^2}{\sum_{B: i\in B}p_Bm_i^2(\omega_B)|B|}=\sum_i\frac{\left[\sum_{\omega\in A_i}p(\omega)m_i(\omega)\right]^2}{\sum_j\sum_{\omega\in A_i\cap A_j}p(\omega)m_i^2(\omega)},
\end{equation}
where the equality holds when $m_i(\omega)=\frac{1}{\deg(\omega)}$ (i.e., $m_i(\omega_B)=\frac{1}{|B|}$), which was first shown by Cohen and Merhav \cite[Theorem 2.1]{Cohen2004}.

When $m_i(\omega)=c_i> 0$, (\ref{temp_using_for_below}) reduces to the DC bound
\begin{equation}\label{}
	P\left(\bigcup_{i=1}^N A_i\right)\ge\sum_i \frac{\left[c_i P(A_i)\right]^2}{\sum_j c_i^2 P(A_i\cap A_j)}=\sum_i\frac{P(A_i)^2}{\sum_j P(A_i\cap A_j)}=\lDC.
\end{equation}

Note that as remarked in \cite{Feng2010}, the DC bound can be seen as a special case of the lower bound
\begin{equation}\label{another_lower_bound}
	P\left(\bigcup_{i=1}^N A_i\right)\ge\frac{\left[\sum_i c_i P(A_i)\right]^2}{\sum_i\sum_j c_i^2 P(A_i\cap A_j)},
\end{equation}
when $c_i=\frac{P(A_i)}{\sum_j P(A_i\cap A_j)}$. This is because
\begin{equation}\label{}
	\begin{split}
		\frac{\left[\sum_i \left(\frac{P(A_i)}{\sum_j P(A_i\cap A_j)}\right) P(A_i)\right]^2}{\sum_i\sum_j \left(\frac{P(A_i)}{\sum_j P(A_i\cap A_j)}\right)^2 P(A_i\cap A_j)}&=\frac{\left(\sum_i\frac{P(A_i)^2}{\sum_j P(A_i\cap A_j)}\right)^2}{\sum_i\left\{\left(\frac{P(A_i)}{\sum_j P(A_i\cap A_j)}\right)^2 \sum_j  P(A_i\cap A_j)\right\}}\\
		&=\frac{\lDC^2}{\lDC}=\lDC.
	\end{split}
\end{equation}
Note that although $c_i>0$ is not assumed in (\ref{another_lower_bound}), one can always replace $c_i$ by $|c_i|$ in (\ref{another_lower_bound}) if $c_i<0$ to get a sharper bound.

However, the lower bound in (\ref{another_lower_bound}) is looser than the following two (left-most) lower bounds (which we later derive in (\ref{new_lower_bound_by_cauchy}) and (\ref{new_temptemp_for_ref_commentsGK})):
\begin{equation}\label{}
	\sum_{i=1}^N \frac{c_i^2 P(A_i)^2}{c_i\sum_{k}c_k P(A_i\cap A_k)}
	\ge \frac{\left[\sum_{i} c_i P(A_i)\right]^2}{\sum_{i}\sum_{k}c_ic_kP(A_i\cap A_k)}\ge\frac{\left[\sum_i c_i P(A_i)\right]^2}{\sum_i\sum_j c_i^2 P(A_i\cap A_j)},
\end{equation}
where $c_i>0$ for all $i$ and the last inequality can be proved using $2c_ic_j\le c_i^2+c_j^2$.
\subsection{Relation to the Gallot-Kounias bound}
By the Cauchy-Schwarz inequality, or assuming $m_i(\omega)=1$ in  (\ref{cauchy_using_m_i}), we have
\begin{equation}\label{}
	\left[\sum_{B: i\in B}f_i(B)p_B\right]\left[\sum_{B: i\in B}\frac{p_B}{f_i(B)}\right]\ge
	\left[\sum_{B: i\in B}p_B\right]^2=P(A_i)^2.
\end{equation}
Using $f_i(B)$ defined using $\bc$ in (\ref{f_def_c}) (note that $f_i(B)>0$ is equivalent to $c_i>0$ for all $i$), we have
\begin{equation}\label{}
	\left[\sum_{B: i\in B}\frac{c_ip_B}{\sum_{k\in B} c_k}\right]\left[\sum_{B: i\in B}\left(\frac{\sum_{k\in B} c_k}{c_i}\right)p_B\right]\ge P(A_i)^2.
\end{equation}
Note that
\begin{equation}\label{}
	\sum_{B: i\in B}\left(\frac{\sum_{k\in B} c_k}{c_i}\right)p_B=\frac{1}{c_i}\sum_{k=1}^N\sum_{B: i\in B, k\in B}c_k p_B=\frac{\sum_{k} c_k P(A_i\cap A_k)}{c_i}.
\end{equation}
Therefore, we have
\begin{equation}\label{}
	\left[\sum_{B: i\in B}\frac{c_ip_B}{\sum_{k\in B} c_k}\right]\left[\frac{\sum_{k} c_k P(A_i\cap A_k)}{c_i}\right]\ge P(A_i)^2.
\end{equation}
Then for all $i$,
\begin{equation}\label{temp_use}
	\sum_{B: i\in B}\frac{c_ip_B}{\sum_{k\in B} c_k}\ge \frac{c_i^2 P(A_i)^2}{c_i\sum_{k}c_k P(A_i\cap A_k)}
\end{equation}
By summing (\ref{temp_use}) over $i$, we get another new lower bound:
\begin{equation}\label{new_lower_bound_by_cauchy}
	P\left(\bigcup_{i}A_i\right)\ge\sum_{i=1}^N \frac{c_i^2 P(A_i)^2}{c_i\sum_{k}c_k P(A_i\cap A_k)}.
\end{equation}
Note that we can use Cauchy-Schwarz inequality again:
\begin{equation}\label{}
	\left[\sum_{i=1}^N \frac{c_i^2 P(A_i)^2}{c_i\sum_{k}c_k P(A_i\cap A_k)}\right]\left[\sum_{i}c_i\sum_{k}c_k P(A_i\cap A_k)\right]\ge \left[\sum_{i}c_iP(A_i)\right]^2,
\end{equation}
which yields
\begin{equation}\label{new_temptemp_for_ref_commentsGK}
	P\left(\bigcup_{i}A_i\right)\ge\sum_{i=1}^N \frac{c_i^2 P(A_i)^2}{c_i\sum_{k}c_k P(A_i\cap A_k)}
	\ge \frac{\left[\sum_{i} c_i P(A_i)\right]^2}{\sum_{i}\sum_{k}c_ic_kP(A_i\cap A_k)}.
\end{equation}
Since the above inequality holds for any positive $\bc$, we have
\begin{equation}\label{inequalities}
	P\left(\bigcup_{i}A_i\right)\ge\max_{\bc\in\mathbb{R}_+^N}\sum_{i=1}^N \frac{c_i^2 P(A_i)^2}{c_i\sum_{k}c_k P(A_i\cap A_k)}
	\ge \max_{\bc\in\mathbb{R}_+^N}\frac{\left[\sum_{i} c_i P(A_i)\right]^2}{\sum_{i}\sum_{k}c_ic_kP(A_i\cap A_k)}.
\end{equation}

One can show that by computing the partial derivative with respect to $c_i$ and set it to zero that
\begin{equation}\label{}
	\max_{\bc\in\mathbb{R}^N}\sum_{i=1}^N \frac{c_i^2 P(A_i)^2}{c_i\sum_{k}c_k P(A_i\cap A_k)}
	=\max_{\bc\in\mathbb{R}^N}\frac{\left[\sum_{i} c_i P(A_i)\right]^2}{\sum_{i}\sum_{k}c_ic_kP(A_i\cap A_k)}=:\lGK,
\end{equation}
where $\lGK$ is the Gallot-Kounias bound (see \cite{Feng2010}), and the optimal $\tilde{\bc}$ can be obtained from 
\begin{equation}
	\bSigma\tilde{\bc}=\balpha,
\end{equation}
where $\balpha=(P(A_1),P(A_2),\dots,P(A_N))^T$ and $\bSigma$ is a $N\times N$ matrix whose $(i,j)$-th element equals to $P(A_i\cap A_j)$. 
Thus, we conclude that the lower bounds in (\ref{inequalities}) are equal to the GK bound as shown in \cite{Feng2010} if $\tilde{\bc}\in\mathbb{R}_+^N$; otherwise, the lower bounds in (\ref{inequalities}) are weaker than the GK bound.

\section{New Bounds using $\boldsymbol{\{P(A_i)\}}$ and $\boldsymbol{\{\sum_j c_jP(A_i\cap A_j)\}}$}\label{sec_3}
\subsection{New Class of Lower Bounds when $\bc$ satisfies (\ref{condition1})}
\begin{theorem}\label{theorem_new_bound_iii_commentsGK}
	For any given $\bc$ that satisfies (\ref{condition1}), a new lower bound on the union probability is given by
	\begin{equation}\label{new_class_lower_bounds}
		\ProbUnion\ge\sum_{i=1}^N\ell_i(\bc)=:\lNEWIII(\bc),
	\end{equation}
	where
	\begin{equation}\label{}
		\begin{split}
			\ell_i(\bc)=P(A_i)\left(\frac{c_i}{\sum_{k\in B_1^{(i)}} c_k}+\frac{c_i}{\sum_{k\in B_2^{(i)}} c_k}\right.
			\left.-\frac{c_i\sum_{k}c_k P(A_i\cap A_k)}{P(A_i)\left(\sum_{k\in B_1^{(i)}} c_k\right)\left(\sum_{k\in B_2^{(i)}} c_k\right)}\right),
		\end{split}
	\end{equation}
	where $B_1^{(i)}$ and $B_2^{(i)}$ are subsets of $\{1,\dots,N\}$ that satisfy the following conditions.
	\begin{enumerate}
		\item If $\frac{\sum_{k}c_k P(A_i\cap A_k)}{c_iP(A_i)}\ge 0$ and $\min_{\{B: i\in B\}}\frac{\sum_{k\in B} c_k}{c_i}<0$, then
		\begin{equation}\label{solution_case1_B}
			\begin{split}
				B_1^{(i)}&=\arg\max_{\{B: i\in B\}}\frac{\sum_{k\in B} c_k}{c_i}\quad\st\quad\frac{\sum_{k\in B} c_k}{c_i}<0,\\
				B_2^{(i)}&=\arg\max_{\{B: i\in B\}}\frac{\sum_{k\in B} c_k}{c_i}.
			\end{split}
		\end{equation}
		\item If $\frac{\sum_{k}c_k P(A_i\cap A_k)}{c_iP(A_i)}\ge 0$ and $\min_{\{B: i\in B\}}\frac{\sum_{k\in B} c_k}{c_i}\ge 0$, then
		\begin{equation}\label{solution_case2_B}
			\begin{split}
				B_1^{(i)}&=\arg\max_{\{B: i\in B\}}\frac{\sum_{k\in B} c_k}{c_i}\quad\\
				&\quad\st\quad \frac{\sum_{k\in B} c_k}{c_i}\le \frac{\sum_{k}c_k P(A_i\cap A_k)}{c_iP(A_i)},\\
				B_2^{(i)}&=\arg\min_{\{B: i\in B\}}\frac{\sum_{k\in B} c_k}{c_i}\\
				&\quad\st\quad \frac{\sum_{k\in B} c_k}{c_i}\ge \frac{\sum_{k}c_k P(A_i\cap A_k)}{c_iP(A_i)}.
			\end{split}
		\end{equation}
		\item If $\frac{\sum_{k}c_k P(A_i\cap A_k)}{c_iP(A_i)}<0$ and $\frac{\sum_{k}c_k P(A_i\cap A_k)}{c_iP(A_i)}< \left\{\max_{\{B: i\in B, \frac{\sum_{k\in B} c_k}{c_i}<0\}}\frac{\sum_{k\in B} c_k}{c_i},\right\}$, then
		\begin{equation}\label{solution_case3_B}
			\begin{split}
				B_1^{(i)}&=\arg\max_{\{B: i\in B\}}\frac{\sum_{k\in B} c_k}{c_i},\quad\st\frac{\sum_{k\in B} c_k}{c_i}<0,\\
				B_2^{(i)}&=\arg\min_{\{B: i\in B\}}\frac{\sum_{k\in B} c_k}{c_i}.
			\end{split}
		\end{equation}
		\item If $\frac{\sum_{k}c_k P(A_i\cap A_k)}{c_iP(A_i)}<0$ and $\frac{\sum_{k}c_k P(A_i\cap A_k)}{c_iP(A_i)}\ge \left\{\max_{\{B: i\in B, \frac{\sum_{k\in B} c_k}{c_i}<0\}}\frac{\sum_{k\in B} c_k}{c_i}\right\},$ then
		\begin{equation}\label{solution_case4_B}
			\begin{split}
				B_1^{(i)}&=\arg\max_{\{B: i\in B\}}\frac{\sum_{k\in B} c_k}{c_i},\\
				B_2^{(i)}&=\arg\max_{\{B: i\in B\}}\frac{\sum_{k\in B} c_k}{c_i}\\
				&\quad\st\quad \frac{\sum_{k\in B} c_k}{c_i}\le \frac{\sum_{k}c_k P(A_i\cap A_k)}{c_iP(A_i)}.
			\end{split}
		\end{equation}
	\end{enumerate}
\end{theorem}
\begin{proof}
Note that for the third and fourth cases, under the condition $\frac{\sum_{k}c_k P(A_i\cap A_k)}{c_iP(A_i)}<0$, the elements of $\bc$ cannot be all positive or negative, so the set $\{B: i\in B, \frac{\sum_{k\in B} c_k}{c_i}<0\}$ is not empty. Therefore, the solutions of $B_1^{(i)}$ and $B_2^{(i)}$ always exist. The proof is given in Appendix \ref{TBA2}.
\end{proof}

\begin{remark}[The new bound ${\lNEWIII(\bc)}$  v.s. the GK bound ${\lGK}$] 	For any $\bc\in\mathbb{R}_+^N$, we have these relations between different lower bounds:
\begin{equation}\label{}
\begin{split}
&\lNEWIII(\bc)\ge  \sum_{i=1}^N \frac{c_i^2 P(A_i)^2}{c_i\sum_{k}c_k P(A_i\cap A_k)}\\
&\ge \frac{\left[\sum_{i} c_i P(A_i)\right]^2}{\sum_{i}\sum_{k}c_ic_kP(A_i\cap A_k)}\ge\frac{\left[\sum_i c_i P(A_i)\right]^2}{\sum_i\sum_j c_i^2 P(A_i\cap A_j)}.
\end{split}
\end{equation}
Note that if $\tilde{\bc}$ obtained by the GK bound satisfies $\tilde{\bc}\in\mathbb{R}_+^N$, then $\lNEWIII(\tilde{\bc})\ge\lGK$. This can be proven by first noting that the two constraints of (\ref{l_i}) and the Cauchy-Schwarz inequality yield (\ref{temp_use}). Then by (\ref{inequalities}), we can get that $\lGK$ is a lower bound of $\lNEWIII(\tilde{\bc})$.	
\end{remark}

\begin{remark}[The new bound ${\lNEWIII(\bc)}$ v.s. the KAT bound ${\lKAT}$]
	One can easily verify that $\lNEWIII(\kappa\bone)=\lKAT$, where $\bone$ is the all-one vector of size $N$ and $\kappa$ is any non-zero constant.
\end{remark}

\begin{lemma}\label{lemma_approx} When $\bc\in\mathbb{R}^N_+$, the lower bound $\lNEWIII(\bc)$ can be computed in pseudo-polynomial time, and can be arbitrarily closely approximated by an algorithm running in polynomial time.
\end{lemma}
\begin{proof}
See Appendix \ref{Pf_lemma_approx}.
\end{proof}

\begin{corollary}{(New class of upper bounds ${\uNEWIV(\bc)}$):}\label{corollary1}
	We can derive an upper bound for any given $\bc\in\mathbb{R}_+^N$ by
	\begin{equation}\label{upper3}
		\begin{split}
			P\left(\bigcup_i A_i\right)&\le\left(\frac{1}{\min_k c_k}+\frac{1}{\sum_k c_k}\right)\sum_i c_i P(A_i)\\
			&\quad-\frac{1}{(\min_k c_k)\sum_k c_k}\sum_i\sum_k c_i c_k P(A_i\cap A_k)=:\uNEWIV(\bc).
		\end{split}
	\end{equation}
	The proof is given in Appendix \ref{Pf_corollary1}.	
	According to the results from randomly generated $\bc$, it is conjectured the optimal upper bound in this class is achieved at $\bc=\kappa\bone$ where $\kappa$ is any non-zero constant.
\end{corollary}

\subsection{New Class of Lower Bounds when $\bc\in\mathbb{R}^N_+$}

We only consider $\bc\in\mathbb{R}^N_+$ in this subsection. A new class of lower bounds, $\lNEWIV$, is given in the following theorem.
\begin{theorem}\label{theoremA}
	Defining $\mathscr{B}^-=\mathscr{B}\setminus \{1,\dots,N\}$, $\tilde{\gamma}_i:=\sum_k c_k P(A_i\cap A_k)$, $\tilde{\alpha}_i:=P(A_i)$ and
	\begin{equation}\label{def_tilde_delta}
		\tilde{\delta}:=\max_{i}\left[\frac{\tilde{\gamma}_i-\left(\sum_k c_k-\min_k c_k\right)\tilde{\alpha}_i}{\min_k c_k}\right]^+,
	\end{equation}
	where $\bc\in\mathbb{R}^N_+$, another class of lower bounds is given by
	\begin{equation}\label{new_class_of_lower_bounds_as_func_of_c}
		\ProbUnion\ge\tilde{\delta}+\sum_{i=1}^N\ell_i'(\bc,\tilde{\delta})=:\lNEWIV(\bc),
	\end{equation}
	where
	\begin{equation}\label{solution_of_l_i_another}
		\begin{split}
			\ell_i'(\bc,x)=&\left[P(A_i)-x\right]\cdot\\
			&\left(\frac{c_i}{\sum_{k\in B_1^{(i)}} c_k}+\frac{c_i}{\sum_{k\in B_2^{(i)}} c_k}-\frac{c_i\sum_{k}c_k \left[P(A_i\cap A_k)-x\right]}{\left[P(A_i)-x\right]\left(\sum_{k\in B_1^{(i)}} c_k\right)\left(\sum_{k\in B_2^{(i)}} c_k\right)}\right),
		\end{split}
	\end{equation}
	and
	\begin{equation}\label{solution_case2_B_improved}
		\begin{split}
			B_1^{(i)}&=\arg\max_{\{B\in\mathscr{B}^-: i\in B\}}\frac{\sum_{k\in B} c_k}{c_i}\quad\\
			&\st\quad \frac{\sum_{k\in B} c_k}{c_i}\le \frac{\sum_{k}c_k \left[P(A_i\cap A_k)-x\right]}{c_i\left[P(A_i)-x\right]},\\
			B_2^{(i)}&=\arg\min_{\{B\in\mathscr{B}^-: i\in B\}}\frac{\sum_{k\in B} c_k}{c_i}\quad\\
			&\st\quad \frac{\sum_{k\in B} c_k}{c_i}\ge \frac{\sum_{k}c_k \left[P(A_i\cap A_k)-x\right]}{c_i\left[P(A_i)-x\right]}.
		\end{split}
	\end{equation}
\end{theorem}
\begin{proof}
Let $x=p_{\{1,2,\dots,N\}}$ and
consider $\sum_{i}\ell_i'(\bc,x)+x$ as a new lower bound where
where $\ell_i'(\bc,x)$  equals to the objective value of the problem
\begin{equation}\label{l_i_another}
	\begin{split}
		&\min_{\{p_B: i\in B, B\in\mathscr{B}^-\}}\sum_{B: i\in B, B\in\mathscr{B}^-}\frac{c_ip_B}{\sum_{k\in B} c_k}\\
		&\st  \sum_{B: i\in B, B\in\mathscr{B}^-}p_B=P(A_i)-x,\\
		&\quad  \sum_{B: i\in B, B\in\mathscr{B}^-}\left(\frac{\sum_{k\in B} c_k}{c_i}\right)p_B=\frac{1}{c_i}\sum_{k}c_k \left[P(A_i\cap A_k)-x\right],\\
		&\qquad p_B\ge 0,\quad \textrm{for all}\quad B\in\mathscr{B}^-\quad\textrm{such that}\quad i\in B.
	\end{split}
\end{equation}
The solution of (\ref{l_i_another}) exists if and only if
\begin{equation}\label{}
	\min_k c_k\le\frac{\tilde{\gamma}_i-(\sum_k c_k)x}{\tilde{\alpha_i}-x}\le\sum_kc_k-\min_k c_k.
\end{equation}
Therefore, the new lower bound can be written as
\begin{equation}\label{improved_lower_bound_new4}
	\begin{split}
		&\min_x\left[x+\sum_{i=1}^N\ell_i'(\bc,x)\right]\quad\st\quad\\
		& \left[\frac{\tilde{\gamma}_i-\left(\sum_k c_k-\min_k c_k\right)\tilde{\alpha}_i}{\min_k c_k}\right]^+\le x\le\frac{\tilde{\gamma}_i-(\min_k c_k)\tilde{\alpha}_i}{\sum_k c_k-\min_k c_k}, \forall i.
	\end{split}
\end{equation}

We can prove that the objective function of (\ref{improved_lower_bound_new4}) is non-decreasing with $x$. Therefore, defining $\tilde{\delta}$ as in (\ref{def_tilde_delta}),
the new lower bound can be written as (\ref{new_class_of_lower_bounds_as_func_of_c})
where $\ell_i'(\bc,\tilde{\delta})$ can be obtained by solving (\ref{l_i_another}), which is given in (\ref{solution_of_l_i_another}). We refer to Appendix \ref{TBA} for more details of the proof.
\end{proof}

\begin{remark}[${\lNEWIV(\bc)}$ v.s. ${\lNEWIII(\bc)}$]
Note that $\lNEWIII(\bc)=\sum_{i=1}^N \ell_i(\bc)$ where $\ell_i(\bc)$ is the solution of (\ref{l_i}). The optimal variable $p_{\{1,\dots,N\}}$ in (\ref{l_i}) is not required to be the same for each $\ell_i(\bc), i=1,\dots,N$. The lower bound $\lNEWIV(\bc)$, however, is the solution of the same problem as for $\lNEWIII(\bc)$ with the additional constraint that the optimal variable $p_{\{1,\dots,N\}}$ in (\ref{l_i}) has the same value for each $\ell_i(\bc), i=1,\dots,N$. Therefore, if $\bc\in\mathbb{R}^N_+$, $\lNEWIII(\bc)$ is the solution of a relaxed problem to the problem for obtaining $\lNEWIV(\bc)$; thus $\lNEWIV(\bc)\ge\lNEWIII(\bc)$.
Also, since $\bc\in\mathbb{R}^N_+$, the solution of (\ref{solution_case2_B_improved}) can be computed in pseudo-polynomial time and has a polynomial-time approximation algorithm.
\end{remark}

\begin{remark}[The new bound ${\lNEWIV(\bc)}$ v.s. the YAT-II bound ${\lNEWII}$]
	One can easily verify that $\lNEWIV(\kappa\bone)=\lNEWII$, where $\bone$ is the all-one vector of size $N$ and $\kappa$ is any non-zero constant.
\end{remark}	

\begin{corollary}{(Improved class of upper bounds ${\uNEWV(\bc)}$):}\label{corollary2}
	We can improve the upper bound $\uNEWIV(\bc)$ in (\ref{upper3}) by 
	\begin{equation}\label{upper4}
		\begin{split}
			P\left(\bigcup_i A_i\right) & \le \min_i\left\{\frac{\sum_k c_k P(A_i\cap A_k)-(\min_k c_k) P(A_i)}{\sum_k c_k-\min_k c_k}\right\} \\
			&+\left(\frac{1}{\min_k c_k}+\frac{1}{\sum_k c_k-\min_k c_k}\right)\sum_i c_i P(A_i)\\
			&-\frac{1}{(\min_k c_k)(\sum_k c_k-\min_k c_k)}\sum_i\sum_k c_i c_k P(A_i\cap A_k),\\
			&=:\uNEWV(\bc).
		\end{split}
	\end{equation}
	Note that the upper bound $\uNEWV(\bc)$ in (\ref{upper4}) is always sharper than $\uNEWIV$ in (\ref{upper3}). The proof is given in Appendix \ref{Pf_corollary2}. According to numerical examples using randomly generated $\bc$, it is conjectured the optimal upper bound in this class is achieved at $\bc=\kappa\bone$, where $\kappa$ is any non-zero constant.
\end{corollary}
\section{Numerical Examples}\label{sec_4}
The same eight systems as in \cite{Yang2014} are used in this section. For comparison, we include bounds that utilize $\{P(A_i)\}$ and $\{\sum_j P(A_i\cap A_j), i=1,\dots,N\}$, such as $\lKAT$, $\lNEWII$ and the optimal lower bound $\lNEWI$ in this class. Furthermore, we included the GK bound $\lGK$ which fully exploit $\{P(A_i)\}$ and $\{P(A_i\cap A_j)\}$ and the PG bound \cite{Prekopa2005}, denoted as $\lPG$, which extends the KAT bound by using $\{P(A_i)\}$, $\{\sum_j P(A_i\cap A_j)\}$ and $\{\sum_{j,l}P(A_i\cap A_j\cap A_l)\}$.

In the numerical examples, $\tilde{\bc}$ is obtained by the GK bound; the elements of $\tilde{\bc}^+$ are given by $\{\tilde{c}_i^+=\max(\tilde{c}_i,\epsilon), i=1,\dots,N\}$ where $\epsilon>0$ is small enough so that if $\tilde{\bc}\in\mathbb{R}^N_+$ then $\tilde{\bc}^+=\tilde{\bc}$.

We present $\lNEWIII(\tilde{\bc})$, $\lNEWIII(\tilde{\bc}^+)$, $\lNEWIV(\tilde{\bc}^+)$ and $\max_{\kappa}\lNEWIII(\tilde{\bc}+\kappa\bone)$ in Table~\ref{table_examples_commentsGK}. In three examples (Systems~II, III and VIII), $\tilde{\bc}\in\mathbb{R}_+^N$; therefore $\lNEWIII(\tilde{\bc})=\lNEWIII(\tilde{\bc}^+)$. In two examples (Systems~VI and VII), $\lNEWIII(\tilde{\bc})$ gives a negative value so we ignore it and replace it by $0$. The lower bound $\max_{\kappa}\lNEWIII(\tilde{\bc}+\kappa\bone)$ is done by searching $\kappa$ from $-1$ to $1$ with a fixed step length $0.005$ (so that $401$ points are used in total). We also randomly generated $100,000$ samples of $\bc\in\mathbb{R}^N_+$ to compute $\lNEWIII(\bc)$ and $\lNEWIV(\bc)$ and the largest bounds were selected and denoted as  $\lNEWIII(\bc_{\textrm{Rand}}^+)$ and $\lNEWIV(\bc_{\textrm{Rand}}^+)$.

From the results, one can see $\lNEWIII(\tilde{\bc}^+)$ is always sharper than $\lNEWIII(\tilde{\bc})$ and is sharper than $\lGK$ in most of the examples except for System~VI. The line search $\max_{\kappa}\lNEWIII(\tilde{\bc}+\kappa\bone)$ is sharper than $\lNEWIII(\tilde{\bc}^+)$ in most of the examples except for System~V. Since $\tilde{\bc}^+\in\mathbb{R}^N_+$, the class of lower bounds $\lNEWIV(\tilde{\bc}^+)$ is always sharper than $\lNEWIII(\tilde{\bc}^+)$, as expected. Furthermore, the PG bound which uses sums of joint probabilities of three events, may be even poorer (e.g., see Systems I and VI) than the numerical bound $\lNEWI$ which utilizes less information but is optimal in the class of lower bounds using $\{P(A_i)\}$ and $\{\sum_j P(A_i\cap A_j)\}$. It is also weaker than the proposed lower bounds in several cases (see Systems~I-IV).

In Table~\ref{table_compare_new34_commentsGK}, we compared $\lNEWIII(\bc)$ and $\lNEWIV(\bc)$ with randomly generated $\bc\in\mathbb{R}^N_+$. One can see that in System~VI, the maximum $\lNEWIV(\bc)$ is $0.3203$ which is sharper than the maximum $\lNEWIII(\bc)$ which is $0.3022$. Also, the percentage that $\lNEWIV(\bc)$ is strictly larger than $\lNEWIII(\bc)$ and the averages of $\frac{\lNEWIV(\bc)}{\lNEWIII(\bc)}$ are shown in Table~\ref{table_compare_new34_commentsGK}.

\section{Conclusion}
In this paper, we present new bounds on the probability of finite union of events using $\boldmath{\{P(A_i)\}}$ and weighted sums of pairwise event probabilities $\boldmath{\{\sum_j c_jP(A_i\cap A_j)\}}$. Two new classes of bounds are proposed which generalize the existing KAT bound and the recently derived YAT bounds. It is also shown that the proposed bounds can be tighter in some cases than the existing GK bound and PG bound which require more information on the events probabilities. These new general union probability bounds can be applied
to effectively estimate and analyze the error performance of
a variety of coded or uncoded communication systems.

\bibliographystyle{AIMS}
\bibliography{Union_Bounds_OR}

\providecommand{\href}[2]{#2}
\providecommand{\arxiv}[1]{\href{http://arxiv.org/abs/#1}{arXiv:#1}}
\providecommand{\url}[1]{\texttt{#1}}
\providecommand{\urlprefix}{URL }
\begin{thebibliography}{10}

\bibitem{Ahmed2013}
\newblock S.~Ahmed and D.~J. Papageorgiou,
\newblock Probabilistic set covering with correlations,
\newblock \emph{Operations Research}, \textbf{61} (2013), 438--452.

\bibitem{Behnamfar2005}
\newblock F.~Behnamfar, F.~Alajaji and T.~Linder,
\newblock Tight error bounds for space-time orthogonal block codes under slow
  {Rayleigh} flat fading,
\newblock \emph{IEEE Transactions on Communications}, \textbf{53} (2005),
  952--956.

\bibitem{Behnamfar2007}
\newblock F.~Behnamfar, F.~Alajaji and T.~Linder,
\newblock An efficient algorithmic lower bound for the error rate of linear
  block codes,
\newblock \emph{IEEE Transactions on Communications}, \textbf{55} (2007),
  1093--1098.

\bibitem{Ben-Tal2009}
\newblock A.~Ben-Tal, L.~El~Ghaoui and A.~Nemirovski,
\newblock \emph{Robust optimization},
\newblock Princeton University Press, 2009.

\bibitem{Bettancourt2008}
\newblock R.~Bettancourt, L.~Szczecinski and R.~Feick,
\newblock {BER} evaluation of {BICM-ID} via {Bonferroni}-type bounds,
\newblock \emph{IEEE Transactions on Vehicular Technology}, \textbf{57} (2008),
  2815--2821.

\bibitem{Cohen2004}
\newblock A.~Cohen and N.~Merhav,
\newblock Lower bounds on the error probability of block codes based on
  improvements on de {C}aen's inequality,
\newblock \emph{IEEE Transactions on Information Theory}, \textbf{50} (2004),
  290--310.

\bibitem{Dawson1967}
\newblock D.~A. Dawson and D.~Sankoff,
\newblock An inequality for probabilities,
\newblock \emph{Proceedings of the American Mathematical Society}, \textbf{18}
  (1967), 504--507.

\bibitem{DeCaen1997}
\newblock D.~De~Caen,
\newblock A lower bound on the probability of a union,
\newblock \emph{Discrete Mathematics}, \textbf{169} (1997), 217--220.

\bibitem{Erdos1959}
\newblock P.~Erd\H{o}s and A.~R\'{e}nyi,
\newblock On {C}antor's series with convergent $\sum 1/q_n$,
\newblock \emph{Ann.\ Univ.\ Sci.\ Budapest. E\H{o}tv\H{o}s Sect. Math.},
  \textbf{2} (1959), 93--109.

\bibitem{Feng2013}
\newblock C.~Feng and L.~Li,
\newblock On the {M}\'{o}ri-{S}z\'{e}kely conjectures for the {Borel-Cantelli}
  lemma,
\newblock \emph{Studia Scientiarum Mathematicarum Hungarica}, \textbf{50}
  (2013), 280--285.

\bibitem{Feng2009}
\newblock C.~Feng, L.~Li and J.~Shen,
\newblock On the {B}orel-{C}antelli lemma and its generalization,
\newblock \emph{Comptes Rendus Mathematique}, \textbf{I} (2009), 1313--1316.

\bibitem{Feng2010}
\newblock C.~Feng, L.~Li and J.~Shen,
\newblock Some inequalities in functional analysis, combinatorics, and
  probability theory,
\newblock \emph{The Electronic Journal of Combinatorics}, \textbf{17} (2010),
  1.

\bibitem{Frolov2012}
\newblock A.~N. Frolov,
\newblock Bounds for probabilities of unions of events and the
  {Borel--Cantelli} lemma,
\newblock \emph{Statistics \& Probability Letters}, \textbf{82} (2012),
  2189--2197.

\bibitem{Gallot1966}
\newblock S.~Gallot,
\newblock A bound for the maximum of a number of random variables,
\newblock \emph{Journal of Applied Probability}, \textbf{3} (1966), 556--558.

\bibitem{Hoppe2006}
\newblock F.~M. Hoppe,
\newblock Improving probability bounds by optimization over subsets,
\newblock \emph{Discrete Mathematics}, \textbf{306} (2006), 526--530.

\bibitem{Hoppe2009}
\newblock F.~M. Hoppe,
\newblock The effect of redundancy on probability bounds,
\newblock \emph{Discrete Mathematics}, \textbf{309} (2009), 123--127.

\bibitem{Kounias1968}
\newblock E.~G. Kounias,
\newblock Bounds for the probability of a union, with applications,
\newblock \emph{The Annals of Mathematical Statistics}, \textbf{39} (1968),
  2154--2158.

\bibitem{Kuai2000}
\newblock H.~Kuai, F.~Alajaji and G.~Takahara,
\newblock A lower bound on the probability of a finite union of events,
\newblock \emph{Discrete Mathematics}, \textbf{215} (2000), 147--158.

\bibitem{Kuai2000a}
\newblock H.~Kuai, F.~Alajaji and G.~Takahara,
\newblock Tight error bounds for nonuniform signaling over {AWGN} channels,
\newblock \emph{IEEE Transactions on Information Theory}, \textbf{46} (2000),
  2712--2718.

\bibitem{Mao2013}
\newblock Z.~Mao, J.~Cheng and J.~Shen,
\newblock A new lower bound on error probability for nonuniform signals over
  {AWGN} channels,
\newblock in \emph{Wireless Communications and Networking Conference (WCNC)},
\newblock IEEE, 2013,
\newblock 3005--3009.

\bibitem{Nemirovski2006}
\newblock A.~Nemirovski and A.~Shapiro,
\newblock Convex approximations of chance constrained programs,
\newblock \emph{SIAM Journal on Optimization}, \textbf{17} (2006), 969--996.

\bibitem{Nguyen2005}
\newblock H.~Nguyen and N.~Tran,
\newblock {Bonferroni}-type bounds for {CDMA} systems with nonuniform
  signalling,
\newblock \emph{IEEE Communications Letters}, \textbf{9} (2005), 583--585.

\bibitem{Ozcelikkale2014}
\newblock A.~Ozcelikkale and T.~M. Duman,
\newblock Lower bounds on the error probability of turbo codes,
\newblock in \emph{IEEE International Symposium on Information Theory (ISIT)},
  2014.

\bibitem{Pinter1989}
\newblock J.~Pint{\'e}r,
\newblock Deterministic approximations of probability inequalities,
\newblock \emph{Zeitschrift f{\"u}r Operations Research}, \textbf{33} (1989),
  219--239.

\bibitem{Prekopa1995}
\newblock A.~Pr{\'e}kopa,
\newblock \emph{Stochastic programming},
\newblock Kluwer Academic Publishers Group, Dordrecht, 1995.

\bibitem{Prekopa2005}
\newblock A.~Pr{\'e}kopa and L.~Gao,
\newblock Bounding the probability of the union of events by aggregation and
  disaggregation in linear programs,
\newblock \emph{Discrete Applied Mathematics}, \textbf{145} (2005), 444--454.

\bibitem{Sasson2006}
\newblock I.~Sasson and S.~Shamai,
\newblock \emph{Performance analysis of linear codes under maximum-likelihood
  decoding: A tutorial},
\newblock Foundations and Trends in Communications and Information Theory, now
  Publishers Inc., 2006.

\bibitem{Seguin1998}
\newblock G.~Seguin,
\newblock A lower bound on the error probability for signals in white
  {Gaussian} noise,
\newblock \emph{IEEE Transactions on Information Theory}, \textbf{44} (1998),
  3168--3175.

\bibitem{Shapiro2014}
\newblock A.~Shapiro, D.~Dentcheva and A.~Ruszczy{\'n}ski,
\newblock \emph{Lectures on stochastic programming: modeling and theory},
  vol.~16,
\newblock SIAM, 2014.

\bibitem{VaziraniBook2001}
\newblock V.~V. Vazirani,
\newblock \emph{Approximation Algorithms},
\newblock Springer-Verlag New York, Inc., New York, NY, USA, 2001.

\bibitem{VenezianiUnpublished}
\newblock P.~Veneziani,
\newblock Lower bounds of degree 2 for the probability of the union of {N}
  events via linear programming, 2007,
\newblock Unpublished.

\bibitem{Yang2014}
\newblock J.~Yang, F.~Alajaji and G.~Takahara,
\newblock Lower bounds on the probability of a finite union of events,
\newblock \urlprefix\url{http://arxiv.org/abs/1401.5543},
\newblock Submitted, 2014.

\bibitem{Yang2014ISIT}
\newblock J.~Yang, F.~Alajaji and G.~Takahara,
\newblock New bounds on the probability of a finite union of events,
\newblock in \emph{2014 IEEE International Symposium on Information Theory
  (ISIT)}, 2014,
\newblock 1271--1275.

\bibitem{Yousefi2004}
\newblock S.~Yousefi and A.~K. Khandani,
\newblock A new upper bound on the {ML} decoding error probability of linear
  binary block codes in {AWGN} interference,
\newblock \emph{IEEE Transactions on Information Theory}, \textbf{50} (2004),
  3026--3036.

\end{thebibliography}

\medskip
\medskip

\clearpage
\begin{table}\caption{Comparison of lower bounds (* indicates $\tilde{\bc}\in\mathbb{R}_+^N$ and a bold number indicates the best results among all tested bounds.)}\label{table_examples_commentsGK}
	\small\small\small
	\centering
	\begin{tabular}{c c c c c c c c c}
		\hline
		System & I & II* & III* & IV & V & VI & VII & VIII*\\ \hline\hline
		$N$ & 6 & 6 & 6 & 7 & 3 & 4 & 4 &4\\ \hline
		$P\left(\bigcup_{i=1}^N A_i\right)$ &  0.7890  &  0.6740  &  0.7890  &  0.9687  &  0.3900 & 0.3252 & 0.5346 & 0.5854\\ \hline
		$\lKAT$ &  0.7247  &  0.6227  &  0.7222  &  0.8909  &  0.3833 & 0.2769 & 0.4434 & 0.5412\\ \hline
		$\lGK$ &  0.7601  &  0.6510  &  0.7508  &  0.9231  &  0.3813 & 0.2972 & 0.4750 & 0.5390\\ \hline
		$\lPG$ & 0.7443 & 0.6434 & 0.7556 & 0.9148 & \textbf{0.3900} & 0.3240 & \textbf{0.5281} & \textbf{0.5726}\\ \hline
		$\lNEWII$ &  0.7247  &  0.6227  &  0.7222  &  0.8909  &  \textbf{0.3900} & 0.3205 & 0.4562 & 0.5464\\  \hline
		$\lNEWI$ &  0.7487  &   0.6398  &  0.7427 &   0.9044  &  \textbf{0.3900} & \textbf{0.3252} & 0.5090 & 0.5531\\ \hline
		$\lNEWIII(\tilde{\bc})$ & 0.6359   & 0.6517*   & 0.7512*  &  0.7908   & 0.3865 & 0 & 0 & 0.5412*\\ \hline
		$\lNEWIII(\tilde{\bc}^+)$ & 0.7638   &  0.6517*  &  0.7512*  &  0.9231   & \textbf{0.3900} & 0.2951 & 0.4905& 0.5412*\\ \hline
		$\lNEWIII(\tilde{\bc}+\kappa\bone)$ & 0.7577   & 0.6539  &  0.7557  &  0.9235 &   0.3899 & 0.2993 & 0.4949 & 0.5412\\ \hline
		$\lNEWIII(\tilde{\bc}^+_{\textrm{Rand}})$ & \textbf{0.7783} & \textbf{0.6633}& \textbf{0.7810} & \textbf{0.9501} & \textbf{0.3900}& 0.3022 & 0.4992 & 0.5666\\ \hline
		$\lNEWIV(\tilde{\bc}^+)$ & 0.7638 & 0.6517 & 0.7512 & 0.9231 & \textbf{0.3900}& 0.2951 & 0.4905 & 0.5412\\ \hline
		$\lNEWIV(\tilde{\bc}^+_{\textrm{Rand}})$ & \textbf{0.7783} & \textbf{0.6633}& \textbf{0.7810} & \textbf{0.9501} & \textbf{0.3900}& 0.3203& 0.4992 & 0.5666 \\ \hline
	\end{tabular}
\end{table}

\clearpage
\begin{table}\caption{Comparison of $\lNEWIII(\bc)$ and $\lNEWIV(\bc)$ with randomly generated $\bc\in\mathbb{R}^N_+$ (a bold number indicates $\max\lNEWIV(\bc)>\max\lNEWIII(\bc)$.)}\label{table_compare_new34_commentsGK}
	\centering
	\begin{tabular}{ccccc}
		\hline
		System & V & VI & VII & VIII*\\ \hline\hline
		$N$  & 3 & 4 & 4 &4\\ \hline
		$P\left(\bigcup_{i=1}^N A_i\right)$ &   0.3900 & 0.3252 & 0.5346 & 0.5854\\ \hline
		$\max \lNEWIII(\bc)$ & 0.3900& 0.3022 & 0.4992 & 0.5666\\ \hline
		$\max \lNEWIV(\bc)$ &  0.3900& \textbf{0.3203} & 0.4992 & 0.5666 \\ \hline
		Average $\frac{\lNEWIV(\bc)}{\lNEWIII(\bc)}$ & 1.0011 & 1.065 & 1.0006 & 1.0000 \\ \hline
		Percentage $\lNEWIV(\bc)>\lNEWIII(\bc)$ & 7.82 \% & 69.6\% & 3.87\% & 0.54 \% \\ \hline
	\end{tabular}
\end{table}

\clearpage
\appendix
	\section{Proof of Theorem \ref{theorem_new_bound_iii_commentsGK}}\label{TBA2}
	We note that $\ell_i(\bc)$ is the solution of
	\begin{equation}\label{l_i}
	\begin{split}
	\min_{\{p_B: i\in B\}}&\sum_{B: i\in B}\frac{c_ip_B}{\sum_{k\in B} c_k}\\
	\st & \sum_{B: i\in B}p_B=P(A_i),\\
	&  \sum_{B: i\in B}\left(\frac{\sum_{k\in B} c_k}{c_i}\right)p_B=\frac{1}{c_i}\sum_{k}c_k P(A_i\cap A_k),\\
	& p_B\ge 0,\quad \textrm{for all}\quad B\in\mathscr{B}\quad\textrm{such that}\quad i\in B.
	\end{split}
	\end{equation}
	From (\ref{l_i}) we have that
	\begin{equation}\label{pure_temp_for_summing_commentsGK}
	\sum_{B: i\in B}\frac{c_ip_B}{\sum_{k\in B} c_k}\ge \ell_i(\bc).
	\end{equation}
	Summing (\ref{pure_temp_for_summing_commentsGK}) over $i$ and using (\ref{prob_union}) we directly obtain
	\begin{equation}\label{}
	P\left(\bigcup_{i=1}^N A_i\right)\ge \sum_{i=1}^N\ell_i(\bc).
	\end{equation}
	
	Note that we can solve (\ref{l_i}) using the same technique used in \cite{Yang2014,Yang2014ISIT}. Consider two subsets $B_1$ and $B_2$ such that $p_{B_1}\ge 0$ and $p_{B_2}\ge 0$, then denoting
	\begin{equation}\label{def_b}
	b:=\frac{\sum_{k}c_k P(A_i\cap A_k)}{c_iP(A_i)}, b_1:=\frac{\sum_{k\in B_1} c_k}{c_i},
	b_2:=\frac{\sum_{k\in B_2} c_k}{c_i},
	\end{equation}
	then problem (\ref{l_i}) reduces to
	\begin{equation}\label{l_i_c_problem}
	\begin{split}
	\ell_i(\bc)=\min_{\{p_{B_1},p_{B_2}\}}\quad &\frac{p_{B_1}}{b_1}+\frac{p_{B_2}}{b_2}\\
	\st & p_{B_1}+p_{B_2}=P(A_i),\\
	& b_1p_{B_1}+b_2p_{B_2}=bP(A_i),\\
	&p_{B_1}\ge 0,\quad p_{B_2}\ge 0.
	\end{split}
	\end{equation}
	According to \cite[Appendix B]{Yang2014}, one can get that
	\begin{equation}\label{l_i_b1_b2}
	\ell_i(\bc)=\min_{\{b_1,b_2: b_1\le b\le b_2\}}\quad P(A_i)\left(\frac{1}{b_1}+\frac{1}{b_2}-\frac{b}{b_1 b_2}\right),
	\end{equation}
	and the partial derivative of $P(A_i)\left(\frac{1}{b_1}+\frac{1}{b_2}-\frac{b}{b_1 b_2}\right)$ with respect to $b_1$ and $b_2$ are (see \cite[Appendix B, Eq.~(B.3)]{Yang2014}):
	\begin{equation}\label{partial_b1_b2}
	\begin{split}
	\frac{\partial\left[P(A_i)\left(\frac{1}{b_1}+\frac{1}{b_2}-\frac{b}{b_1 b_2}\right)\right]}{\partial b_1}&=\frac{P(A_i)}{b_1^2}\left(\frac{b-b_2}{b_2}\right),\quad \\ \frac{\partial\left[P(A_i)\left(\frac{1}{b_1}+\frac{1}{b_2}-\frac{b}{b_1 b_2}\right)\right]}{\partial b_2}&=\frac{P(A_i)}{b_2^2}\left(\frac{b-b_1}{b_1}\right).
	\end{split}
	\end{equation}
	Note that the partial derivatives are not continuous at $b_1=0$ and $b_2=0$.
	Therefore, the solution depends on the following different scenarios.
	\begin{enumerate}
		\item If $b\ge 0$ and $$\min_{\{B: i\in B\}}\frac{\sum_{k\in B} c_k}{c_i}<0,$$ the solutions of (\ref{l_i_b1_b2}) are given by
		\begin{equation}\label{solution_case1}
		\begin{split}
		b_1&=\max_{\{B: i\in B\}}\frac{\sum_{k\in B} c_k}{c_i}\quad\st\quad\frac{\sum_{k\in B} c_k}{c_i}<0,\\
		b_2&=\max_{\{B: i\in B\}}\frac{\sum_{k\in B} c_k}{c_i}.
		\end{split}
		\end{equation}
		\item If $b\ge 0$ and $$\min_{\{B: i\in B\}}\frac{\sum_{k\in B} c_k}{c_i}\ge 0,$$ the solutions of (\ref{l_i_b1_b2}) are given by
		\begin{equation}\label{solution_case2}
		\begin{split}
		b_1&=\max_{\{B: i\in B\}}\frac{\sum_{k\in B} c_k}{c_i}\quad\st\quad \frac{\sum_{k\in B} c_k}{c_i}\le b,\\
		b_2&=\min_{\{B: i\in B\}}\frac{\sum_{k\in B} c_k}{c_i}\quad\st\quad \frac{\sum_{k\in B} c_k}{c_i}\ge b.
		\end{split}
		\end{equation}
		\item If $b<0$ and $$b< \left\{\max_{\{B: i\in B\}}\frac{\sum_{k\in B} c_k}{c_i},\quad\st\frac{\sum_{k\in B} c_k}{c_i}<0\right\},$$ the solutions of (\ref{l_i_b1_b2}) are given by
		\begin{equation}\label{solution_case3}
		\begin{split}
		b_1&=\max_{\{B: i\in B\}}\frac{\sum_{k\in B} c_k}{c_i},\quad\st\frac{\sum_{k\in B} c_k}{c_i}<0,\\
		b_2&=\min_{\{B: i\in B\}}\frac{\sum_{k\in B} c_k}{c_i}.
		\end{split}
		\end{equation}
		\item If $b<0$ and $$b\ge \left\{\max_{\{B: i\in B\}}\frac{\sum_{k\in B} c_k}{c_i},\quad\st\frac{\sum_{k\in B} c_k}{c_i}<0\right\},$$ the solutions of (\ref{l_i_b1_b2}) are given by
		\begin{equation}\label{solution_case4}
		\begin{split}
		b_1&=\max_{\{B: i\in B\}}\frac{\sum_{k\in B} c_k}{c_i},\\
		b_2&=\max_{\{B: i\in B\}}\frac{\sum_{k\in B} c_k}{c_i}\quad\st\quad \frac{\sum_{k\in B} c_k}{c_i}\le b.
		\end{split}
		\end{equation}
	\end{enumerate}
	
	\section{Proof of Lemma \ref{lemma_approx}}\label{Pf_lemma_approx}
	The problems in (\ref{solution_case1_B}) to (\ref{solution_case4_B}) are exactly the $0/1$ knapsack problem with mass equals to value (see \cite{VaziraniBook2001}, the corresponding decision problem is also called subset sum problem). Unfortunately, the $0/1$ knapsack problem is NP-hard in general.
	
	However, if $\bc\in\mathbb{R}^N_+$, i.e, the case (\ref{solution_case2_B}), there exists a dynamic programming solution which runs in pseudo-polynomial time, i.e., polynomial in $N$, but exponential in the number of bits required to represent $\frac{\sum_{k}c_k P(A_i\cap A_k)}{c_iP(A_i)}$ (see \cite{VaziraniBook2001}). Furthermore, there is a fully polynomial-time approximation scheme (FPTAS), which finds a solution that is correct within a factor of $(1-\epsilon)$ of the optimal solution (see \cite{VaziraniBook2001}). The running time is bounded by a polynomial and $1/\epsilon$ where $\epsilon$ is a bound on the correctness of the solution.
	
	Therefore, if $\bc\in\mathbb{R}^N_+$, one can get a lower bound for $\ell_i(\bc)$ in polynomial time which can be arbitrarily close to $\ell_i(\bc)$ by setting $\epsilon$ small enough, i.e.,
	\begin{equation}\label{}
	\ell_i(\bc)\ge \ell_i^L(\bc,\epsilon),\quad \lim_{\epsilon\rightarrow 0^+}\ell_i^L(\bc,\epsilon)=\ell_i(\bc).
	\end{equation}
	The details are as follows. First, assume $\hat{B}_1$ and $\hat{B}_2$ are obtained by the FPTAS which satisfy
	\begin{equation}\label{}
	(1-\epsilon) \sum_{k\in B_1^{(i)}} c_k \le \sum_{k\in \hat{B}_1^{(i)}} c_k \le \sum_{k\in B_1^{(i)}} c_k,\quad \sum_{k\in B_2^{(i)}} c_k\le \sum_{k\in \hat{B}_2^{(i)}} c_k \le (1+\epsilon) \sum_{k\in B_2^{(i)}} c_k.
	\end{equation}
	Then we have
	\begin{equation}\label{}
	\begin{split}
	\sum_{k\in B_1^{(i)}} c_k&\le \min\left\{\frac{\sum_{k\in \hat{B}_1^{(i)}} c_k}{1-\epsilon}, \frac{\sum_{k}c_k P(A_i\cap A_k)}{P(A_i)}\right\}=: b_1^{(i)}, \\
	\sum_{k\in B_2^{(i)}} c_k&\ge \max\left\{\frac{\sum_{k\in B_2^{(i)}} c_k}{1+\epsilon}, \frac{\sum_{k}c_k P(A_i\cap A_k)}{P(A_i)}\right\}=: b_2^{(i)}.
	\end{split}
	\end{equation}
	Then one can get the arbitrarily close lower bound for $\ell_i(\bc)$ as
	\begin{equation}\label{}
	\ell_i(\bc)\ge \ell_i^L(\bc,\epsilon):=P(A_i)\left(\frac{c_i}{b_1^{(i)}}+\frac{c_i}{b_2^{(i)}} -\frac{c_i\sum_{k}c_k P(A_i\cap A_k)}{P(A_i)b_1^{(i)}b_2^{(i)}}\right).
	\end{equation}
	
	Therefore, we can get a lower bound for $\ProbUnion$ that is arbitrarily close to $\lNEWIII(\bc)$ in polynomial time:
	\begin{equation}\label{}
	\ProbUnion\ge\sum_i\ell_i(\bc)\ge\sum_i\ell_i^L(\bc,\epsilon).
	\end{equation}

	\section{Proof of Corollary \ref{corollary1}}\label{Pf_corollary1}
	We get the upper bound by maximizing, instead of minimizing, the objective function of (\ref{l_i}). More specifically, for any given $\bc\in\mathbb{R}^+$, a upper bound can be obtained by
	\begin{equation}\label{new_upper}
	\hbar(\bc)=\sum_{i=1}^N\hbar_i(\bc),
	\end{equation}
	where $\hbar_i(\bc)$ is defined by
	\begin{equation}\label{l_i_max}
	\begin{split}
	\hbar_i(\bc):=\max_{\{p_B: i\in B\}}&\sum_{B: i\in B}\frac{c_ip_B}{\sum_{k\in B} c_k}\\
	\st & \sum_{B: i\in B}p_B=P(A_i),\\
	&  \sum_{B: i\in B}\left(\frac{\sum_{k\in B} c_k}{c_i}\right)p_B=\frac{1}{c_i}\sum_{k}c_k P(A_i\cap A_k),\\
	& p_B\ge 0,\quad \textrm{for all}\quad B\in\mathscr{B}\quad\textrm{such that}\quad i\in B.
	\end{split}
	\end{equation}
	The resulting upper bound is given by
	\begin{equation}\label{upper3-proof}
	\begin{split}
	P\left(\bigcup_i A_i\right)&\le\sum_i \left\{P(A_i)\left[\frac{c_i}{\min_k c_k}+\frac{c_i}{\sum_k c_k}-\frac{c_i^2}{(\min_k c_k) \sum_k c_k}\frac{\sum_k c_k P(A_i\cap A_k)}{c_i P(A_i)}\right]\right\}\\
	&=\left(\frac{1}{\min_k c_k}+\frac{1}{\sum_k c_k}\right)\sum_i c_i P(A_i)-\frac{1}{(\min_k c_k)\sum_k c_k}\sum_i\sum_k c_i c_k P(A_i\cap A_k).
	\end{split}
	\end{equation}

	\section{Proof of Theorem \ref{theoremA}}\label{TBA}
	Let $x=p_{\{1,\dots,N\}}$ and define $\mathscr{B}^-=\mathscr{B}\setminus \{1,\dots,N\}$, then
	consider $\sum_{i}\ell_i'(\bc,x)+x$ as a new lower bound
	where $\ell_i'(\bc,x)$ is defined by the solution of (\ref{l_i_another}),
	which exists if and only if
	\begin{equation}\label{}
	\min_k c_k\le\frac{\tilde{\gamma}_i-(\sum_k c_k)x}{\tilde{\alpha_i}-x}\le\sum_kc_k-\min_k c_k,
	\end{equation}
	which gives
	\begin{equation}\label{}
	\begin{split}
	&\max_i\left[\frac{\tilde{\gamma}_i-\left(\sum_k c_k-\min_k c_k\right)\tilde{\alpha}_i}{\min_k c_k}\right]^+\le x\le\min_i\left[\frac{\tilde{\gamma}_i-(\min_k c_k)\tilde{\alpha}_i}{\sum_k c_k-\min_k c_k}\right].
	\end{split}
	\end{equation}
	Therefore, the new lower bound can be written as (\ref{improved_lower_bound_new4}).
	
	Next, we can prove that the objective function of (\ref{improved_lower_bound_new4}) is non-decreasing with $x$. First, we prove
	\begin{equation}\label{}
	\begin{split}
	\ell_i'(\bc,x)=&\left[P(A_i)-x\right]\\
	&\left(\frac{c_i}{\sum_{k\in B_1^{(i)}} c_k}+\frac{c_i}{\sum_{k\in B_2^{(i)}} c_k}-\frac{c_i\sum_{k}c_k \left[P(A_i\cap A_k)-x\right]}{\left[P(A_i)-x\right]\left(\sum_{k\in B_1^{(i)}} c_k\right)\left(\sum_{k\in B_2^{(i)}} c_k\right)}\right),
	\end{split}
	\end{equation}
	is continuous when $\exists B'\in\mathscr{B}^-$ such that
	\begin{equation}\label{}
	\frac{\sum_{k}c_k \left[P(A_i\cap A_k)-x\right]}{c_i\left[P(A_i)-x\right]}=\frac{\sum_{k\in B' }c_k}{c_i}.
	\end{equation}
	This can be proved by
	\begin{equation}\label{}
	\lim_{h\rightarrow 0^+} \ell_i'(\bc,x+h)=\lim_{h\rightarrow 0^+} \ell_i'(\bc,x-h)=\frac{c_i}{\sum_{k\in B' }c_k}.
	\end{equation}
	Then one can prove that when
	\begin{equation}\label{}
	\frac{\sum_{k\in B_2^{(i)} }c_k}{c_i}< \frac{\sum_{k}c_k \left[P(A_i\cap A_k)-x\right]}{c_i\left[P(A_i)-x\right]}<\frac{\sum_{k\in B_1^{(i)} }c_k}{c_i},
	\end{equation}
	the partial derivative of $\ell_i'(\bc,x)+\frac{c_i}{\sum_k c_k}x$ w.r.t. $x$ is non-negative:
	\begin{equation}\label{}
	\begin{split}
	&\frac{\partial \left(\ell_i'(\bc,x)+\frac{c_i}{\sum_k c_k}x\right)}{\partial x}\\
	&=\frac{c_i}{\sum_k c_k}-\frac{c_i}{\sum_{k\in B_1^{(i)} }c_k}-\frac{c_i}{\sum_{k\in B_2^{(i)} }c_k}\\
	&\quad+\frac{c_i\sum_k c_k}{\left(\sum_{k\in B_1^{(i)}} c_k\right)\left(\sum_{k\in B_2^{(i)}} c_k\right)}\\
	&=\frac{c_i\left(\sum_k c_k -\sum_{k\in B_1^{(i)}} c_k\right)\left(\sum_k c_k -\sum_{k\in B_2^{(i)}} c_k\right)}{\left(\sum_k c_k\right)\left(\sum_{k\in B_1^{(i)}} c_k\right)\left(\sum_{k\in B_2^{(i)}} c_k\right)}\\
	&=\frac{c_i\left(\sum_{k\notin B_1^{(i)}} c_k\right)\left(\sum_{k\notin B_2^{(i)}} c_k\right)}{\left(\sum_k c_k\right)\left(\sum_{k\in B_1^{(i)}} c_k\right)\left(\sum_{k\in B_2^{(i)}} c_k\right)}\ge 0.
	\end{split}
	\end{equation}
	Therefore, the objective function of (\ref{improved_lower_bound_new4}),
	\begin{equation}\label{}
	\sum_i \ell_i'(\bc,x)+x=\sum_i\left(\ell_i'(\bc,x)+\frac{c_i}{\sum_k c_k}x\right),
	\end{equation}
	is non-decreasing with $x$.

	Finally, defining $\tilde{\delta}$ as in (\ref{def_tilde_delta}),
	the new lower bound can be written as
	\begin{equation}\label{}
	\ProbUnion\ge\tilde{\delta}+\sum_{i=1}^N\ell_i'(\bc,\tilde{\delta}),
	\end{equation}
	where $\ell_i'(\bc,\tilde{\delta})$ can be obtained using the solution for $\ell_i(\bc)$ in Theorem \ref{theorem_new_bound_iii_commentsGK} with $b=\frac{\sum_{k}c_k P(A_i\cap A_k)}{c_iP(A_i)}$ replaced by $\tilde{b}=\frac{\sum_{k}c_k \left[P(A_i\cap A_k)-\tilde{\delta}\right]}{c_i\left[P(A_i)-\tilde{\delta}\right]}$.

	\section{Proof of Corollary \ref{corollary2}}\label{Pf_corollary2}
	Letting $x=p_{\{1,\dots,N\}}$. Defining $\mathscr{B}^-=\mathscr{B}\setminus \{1,\dots,N\}$, then
	\begin{equation}\label{new_upper}
	\hbar'(\bc)=\max_x\left[x+\sum_{i=1}^N\hbar_i'(\bc,x)\right],
	\end{equation}
	where $\hbar_i'(\bc,x)$ is defined by
	\begin{equation}\label{l_i_max2}
	\begin{split}
	\hbar_i'(\bc,x):=\max_{\{p_B: i\in B, B\in\mathscr{B}^-}&\sum_{B: i\in B, B\in\mathscr{B}^-}\frac{c_ip_B}{\sum_{k\in B} c_k}\\
	\st & \sum_{B: i\in B, B\in\mathscr{B}^-}p_B=P(A_i)-x,\\
	&  \sum_{B: i\in B, B\in\mathscr{B}^-}\left(\frac{\sum_{k\in B} c_k}{c_i}\right)p_B=\frac{1}{c_i}\sum_{k}c_k \left[P(A_i\cap A_k)-x\right],\\
	& p_B\ge 0,\quad \textrm{for all}\quad B\in\mathscr{B}^-\quad\textrm{such that}\quad i\in B.
	\end{split}
	\end{equation}
	The solution of $\hbar_i'(\bc,x)$ is independent with $x$:
	\begin{equation}\label{}
	\begin{split}
	\hbar_i'(\bc,x)
	&=\left(P(A_i)-x\right)\left(\frac{c_i}{\min_k c_k}+\frac{c_i}{\sum_k c_k-\min_k c_k}\right)\\
	&\qquad-\frac{c_i}{(\min_k c_k) (\sum_k c_k-\min_k c_k)}\sum_k c_k \left(P(A_i\cap A_k)-x\right),\\
	&=P(A_i)\left(\frac{c_i}{\min_k c_k}+\frac{c_i}{\sum_k c_k-\min_k c_k}\right)\\
	&\qquad-\frac{c_i}{(\min_k c_k) (\sum_k c_k-\min_k c_k)}\sum_k c_k P(A_i\cap A_k),
	\end{split}
	\end{equation}
	and the solution exists if and only if for all $i$
	\begin{equation}\label{}
	\min_k c_k\le\frac{\sum_k c_k P(A_i\cap A_k) -(\sum_k c_k) x}{P(A_i)-x}\le \sum_k c_k-\min_k c_k.
	\end{equation}
	Thus, we get
	\begin{equation}\label{}
	\begin{split}
	&\left\{\max_i\frac{\sum_k c_k P(A_i\cap A_k)-\left(\sum_k c_k-\min_k c_k\right)P(A_i)}{\min_k c_k}\right\}^+\\
	&\le x\le\min_i\frac{\sum_k c_k P(A_i\cap A_k)-(\min_k c_k) P(A_i)}{\sum_k c_k-\min_k c_k}
	\end{split}
	\end{equation}
	Therefore, we get the upper bound
	\begin{equation}\label{upper4-proof}
	\begin{split}
	P\left(\bigcup_i A_i\right) & \le \min_i\left\{\frac{\sum_k c_k P(A_i\cap A_k)-(\min_k c_k) P(A_i)}{\sum_k c_k-\min_k c_k}\right\} \\
	&+\left(\frac{1}{\min_k c_k}+\frac{1}{\sum_k c_k-\min_k c_k}\right)\sum_i c_i P(A_i)\\
	&-\frac{1}{(\min_k c_k)(\sum_k c_k-\min_k c_k)}\sum_i\sum_k c_i c_k P(A_i\cap A_k).
	\end{split}
	\end{equation}

\end{document}